\documentclass[12pt]{amsart}

\usepackage{amstext, amsmath, amssymb, amsfonts, amsthm, mathtools, latexsym}
\usepackage{graphicx, tikz}
\usepackage{hyperref}
\usepackage[all,cmtip]{xy}
\usepackage[left=3cm, right=3cm, bottom=4cm]{geometry} 

\begin{document}

\newtheoremstyle{basico}
  {0,2cm}
  {0}
  {\itshape}
  {0,5cm}
  {\bfseries}
  {}
  {0,2cm}
  {\thmname{#1}\thmnumber{ #2}\thmnote{ #3}:}
\theoremstyle{basico}  

\newtheorem{teoprin}{Theorem}  
\newtheorem{coro}{Corollary}
\newtheorem{teo}{Theorem}[section]
\newtheorem{lema}[teo]{Lemma}
\newtheorem{prop}[teo]{Proposition}
\newtheorem{defi} [teo]  {Definition}
\newtheorem{rem}[teo]{Remark}
\newtheorem{afirm}{Afirmation}

\renewcommand\theteoprin{\Alph{teoprin}}

\newtheoremstyle{ejemplos}
  {0,2cm}
  {0}
  {}
  {0,5cm}
  {\bfseries}
  {}
  {0,2cm}
  {\thmname{#1}\thmnumber{ #2}:\thmnote{ #3}}
\theoremstyle{ejemplos}
\newtheorem{ex}{Example}
\newtheorem{quest}{Question}

\newcommand{\D}{{\mathbb D}}
\newcommand{\N} {\mathbb N}

\newcommand{\A}{\mathcal{A}}
\newcommand{\M}{\mathcal{M}}
\renewcommand{\P}{\mathcal{P}}
\newcommand{\X}{\mathcal{X}}
\newcommand{\Z} {\mathbb Z}
\newcommand{\R} {\mathbb R}
\newcommand{\e} {\varepsilon}

\newcommand{\card}{\operatorname{card}}
\newcommand{\interior}{\operatorname{int}}
\newcommand{\deter}{\operatorname{det}}
\newcommand{\supp}{\operatorname{supp}}
\newcommand{\ind}{\operatorname{ind}}
\newcommand{\per}{\operatorname{Per}}
\newcommand{\dif}{\operatorname{Diff}}
\newcommand{\emb}{\operatorname{Emb^r}}

\title{ On the tree models of mildly dissipative maps}

\author{Javier Correa}
\address{Universidade Federal de Minas Gerais}
\email {jcorrea@mat.ufmg.br}

\author{Elizabeth Flores} 
\address{Universidad Nacional Mayor de San Marcos} 
\email {esalazarf@unmsm.edu.pe}

\thanks{The first author has been supported by CAPES and CNPq, and would like to thank UFRJ since this work started at his postdoctoral position in said university. The second author has been supported by CAPES and CNPq.} 

\subjclass[2020]{37C15, 37D25, 37D45, 37E25}

\keywords{surface diffeomorphisms, mild dissipation, dynamical invariants}

\begin{abstract} {This study examines the tree models of mildly dissipative diffeomorphisms on the disk $\D$. These models are one-dimensional dynamical systems with ergodic aperiodic data as well as some properties of the original dynamics. The focus of this work is their quality as dynamical invariants in both the topological and ergodic sense.}
\end{abstract}

\maketitle

\section{Introduction}	
A classical way to study dynamical systems is by searching reduced models that capture the main features of the object of study. S. Crovisier and E. Pujals introduced a one-dimensional model in \cite{pujcrov} to study a family of surface diffeomorphisms that they labeled as strongly dissipative and was later renamed to mildly dissipative in \cite{CPT}. The current study examines this model from the perspective of a dynamical invariant. This inquiry is natural as well as interesting because these models are topological objects, yet they are built from an ergodic standpoint of differentiable maps.

The complexity of mildly dissipative diffeomorphisms exists in between one-dimensional dynamics and surface diffeomorphisms. Let us briefly recall how they are defined. Consider the compact disc $\D$ and given $ r \geq 1 $, we use $\emb(\D)$ to denote the space of $C^r$ embeddings of $\D$ into itself. We say $f\in \emb(\D)$ is dissipative if  $ |\det Df (x) |<1, \text{for all } x \in \D $. This condition implies that for every invariant measure, for almost every point $x$, there is a stable manifold $W^s(x)$. If the measure is not supported on a hyperbolic sink, then the stable manifold of these points must have dimension 1. Let us call $W^s_\D(x)$ the connected component of $W^s(x)\cap \D$ that contains $x$. We say that $f$ is mildly dissipative if it is dissipative and if $W^s_\D(x)$ splits the disk in two sets for every $x$ that has a one-dimensional stable manifold. We denote the family of such maps as $\text{MD}^r(\D)$.

A simple example of a mildly dissipative map is the classical construction of the horseshoe in the disc $\D$. Another family of examples is the Henon maps $H_{a,b}(x,y)=(1-ax^2+y,-bx)$ (introduced in \cite{Hen}) for certain parameters $a$ and $b$ (see \cite{pujcrov} for more details). Moreover, Boro\'nski and \v{S}timac in \cite{borosti} show mild dissipation for certain parameters of Lozi maps (introduced in \cite{Lozi}) $L_{a,b}(x,y)=(1+y-a|x|,bx)$ . 

Some of the main features of mildly dissipative diffeomorphisms (proved in \cite{pujcrov}) are as follows:
\begin{enumerate}
\item The interior of $\text{MD}^r(\D)$ is not empty. When the Jacobian is small enough, mild dissipation becomes a $C^1$-open property. 
\item The periodic points are dense in the support of every non-atomic ergodic measure. This result is obtained through a $C^\infty$ closing lemma for invariant measures without any perturbation. 
\end{enumerate}

We encourage the reader to refer to \cite{borosti}, \cite{borosti2}, \cite{pujcrov}, \cite{CP02}, \cite{CPT},  and \cite{CLPY} for a better understanding of mildly dissipative diffeomorphisms. 

To prove the second item of the previous statements, the authors define a real tree and a dynamical system on it as a reduced model of the former mildly dissipative map. Roughly speaking, the tree is defined by considering the quotient of the disk $\D$ by the curves $W^s_\D(x)$. Indeed, there are some technicalities to consider, and we discuss them in detail in section \ref{treeConstruction}.

Given a compact topological space $X$ and continuous map $f:X\to X$, we define $\M_1(f)$ as the set of invariant probability measures associated to $f$ and $\M_a(f)$ as the set of ergodic aperiodic measures of $f$. 

\begin{teo}[Crovisier--Pujals]\label{teoCP01}
Consider $f\in\text{MD}^r (\D) $, with $r>1$, such that $\M_a(f)\neq \emptyset$. Then, there exists a compact real tree $\X$, two continuous maps $\hat{f}:\X\to\X$, and  $\pi:\D\rightarrow \X$ that verify the following:
\begin{enumerate}
\item $\pi$ is a surjective semi-conjugacy between $f$ and $\hat{f}$, i.e., $\pi \circ f = \hat{f}\circ \pi$. 
\item The push-forward map induced by $\pi$, $\pi_*: \M_1(f)\to \M_1(\hat{f})$  is injective on $\M_a(f)$.
\end{enumerate}
\end{teo}

Our first contribution to the study of such objects is two extra properties.

\begin{prop}\label{PropProp}
Consider $f\in \text{MD}^r (\D)$ and $(\X,\hat{f})$ as in Theorem \ref{teoCP01}. Then,
\begin{enumerate}
\item	$\pi(\per(f))= \per(\hat{f})$
\item $\pi_*:\M_1(f)\to \M_1(\hat{f})$ is surjective. 
\end{enumerate}
\end{prop}

We would like to make a few comments. First, in \cite{borosti}, for Wang--Young parameters in the Henon family (see \cite{WY}) and for the Misiurewicz--\v{S}timac parameters of Lozi maps (see \cite{MS}), the authors show that branching points are dense and the trees seem to have a fractal structure. Second, $\pi_*$ might not be a bijection. We do have that it is surjective, injective on $\M_a(f)$, $\pi(\per(f))= \per(\hat{f})$, and $\pi_*$ is a convex map. However, it may fail to be injective in the set of periodic measures, and there are simple examples of this.
	
We now proceed to study the induced tree of a mildly dissipative map as a dynamical invariant.

\begin{quest} \label{QuestInv} Given $f,g\in \text{MD}^r (\D)$, consider $(\X_f,\hat{f})$ and $(\X_g,\hat{g})$ associated to $f$ and $g$ respectively. If the dynamics of $f$ and $g$ are equivalent, are the dynamics of $\hat{f}$ and $\hat{g}$ also equivalent? 
\end{quest}

We call a tree model of $f\in \text{MD}^r (\D)$ to any triple $(\X,\pi:\D\to \X, k:\X\to \X)$ such that
\begin{enumerate}
\item $\pi$ is a surjective semi-conjugacy between $f$ and $k$, and
\item  $\pi$ induces a map $\pi_*: \M_1(f)\to \M_1(k)$ that is surjective and injective on $\M_a(f)$.
\end{enumerate}

\begin{quest}\label{QuestTreeUni} Given $f\in \text{MD}^r (\D)$, is there a unique tree model? 
\end{quest}

To avoid confusion, when considering $(\X,\pi,\hat{f})$ or just $(\X,\hat{f})$, we are referring to the tree model constructed in Theorem \ref{teoCP01}  (and not a general one), and we call it the natural tree model of $f$.

We study the first question in a topological sense, and its answer depends on where the conjugacy is defined. If there exists a homeomorphism $h:\D\to \D$ that conjugates $f$ and $g$, then the answer is yes. The map $h$ induces a homeomorphism from $\X_f$ to $\X_g$ that conjugates $\hat{f}$ with $\hat{g}$. 

\begin{prop}\label{ConjDisco}
	Let $f$  and $g$ in  $ \text{MD}^r (\D)$ and suppose there exists a homeomorphism $h:\D\to \D$ that verifies $g\circ h = h \circ f$. Then, there exists a homeomorphism $\hat h:\X_f\to \X_g$ such that $\hat g\circ \hat h = \hat h \circ \hat f$. 
\end{prop}

In \cite{borosti2} (a work contemporary to this one), a similar result is proved for the Wang--Young parameters in the Henon family.

The condition of $h$ to be defined in the whole disk is too restrictive. This hypothesis gives information on the dynamics in the wandering set and in particular implies that stable manifolds go into stable manifolds. However, concerning our second question, we observe that none of the mentioned properties are related to the wandering set. Therefore, it is natural to wonder what happens when $h$ is only defined from the maximal invariant set of $f$ to the maximal invariant set of $g$. 

\begin{ex}\label{contra}
There exists $f$  and $g$ in  $ \text{MD}^r (\D)$ such that, if $\Lambda_f$ and $\Lambda_g$ are the maximal invariant sets of $f$ and $g$ respectively in $\D$, then,
\begin{enumerate} 
\item there exists a homeomorphism $h:\Lambda_f\to \Lambda_g$ that conjugates $f_{|_{\Lambda_f}}$ with $g_{|_{\Lambda_g}}$,
\item $\X_f$ is an interval, and
\item $\X_g$ has one point of index $3$. 
\end{enumerate}
\end{ex} 

Since $\X_f$ and $\X_g$ cannot be homeomorphic, the first question gets a negative answer when the conjugacy is only defined on the maximal invariant set. In this example, although $\X_f$ and $\X_g$ are not the same tree, continuous and surjective semi-conjugacies can be constructed in both directions. The existence of these semi-conjugacies is the best one can hope for, and our main theorem states that this always happens. 

Before enunciating our first theorem, some troublesome wandering dynamics of $f$ must be removed.  Let $f$ be a mildly dissipative map and $\Lambda_f$ its maximal invariant set. Then, we say that $f$ is stable efficient if $\pi(\Lambda_f) = \X_f$. A second hypothesis we need for $f$ to verify is what we call to have finite ergodic covering, and it means that we can construct the tree $\X_f$ using only a finite number of ergodic aperiodic measures. However, we would like to point out that this hypothesis does not imply the set $\M_a(f)$ to be finite. A precise definition of this concept is provided in section \ref{treeConstruction}.

\begin{teoprin}\label{TeoSemi}
Let $f,g$ in  $\text{MD}^r (\D)$ be stable efficient with  finite ergodic covering and $\Lambda_f$ and $\Lambda_g$ be the maximal invariant sets of $f$ and $g$ respectively in $\D$. Consider $(\X_f,\hat f)$ and $(\X_g, \hat g)$ the natural tree models associated to $f$ and $g$. Suppose there exists a homeomorphism $h:\Lambda_f\to \Lambda_g$ verifying $h\circ f_{|_{\Lambda_f}} = g_{|_{\Lambda_g}} \circ h$. Then, there exist two continuous surjective maps $\hat h_1:\X_f\rightarrow \X_g$ and $\hat h_2:\X_g\rightarrow \X_f$ such that $\hat h_1\circ\hat{f}=\hat{g}\circ \hat h_1$ and  $\hat h_2\circ\hat{g}=\hat{f}\circ \hat h_2$.  	
\end{teoprin}

Since both maps in Example \ref{contra} are stable efficient and have finite ergodic covering, applying the previous theorem, we conclude that  $(\X_f,\hat f)$ and $(\X_g, \hat g)$ are tree models of the same map. Therefore, we conclude a negative answer to our second question. 

 We observe that the stable efficiency property is a necessary condition, and our next proposition shows that it is not restrictive in a meaningful way. 

\begin{prop}\label{PropStabEff}
Let $(f,\D)$ be a mildly dissipative map and $\Lambda_f$ its maximal invariant set. If $f$ is not stable efficient, then there exists $D\subset \D$ that is homeomorphic to $\D$ such that:
\begin{enumerate}
\item $(f_{|_D}, D)$ is mildly dissipative and stable efficient. 
\item The maximal invariant set of $f$ in $D$ is  the maximal invariant set of $f$ in $\D$.
\item The tree associated to $f_{|_D}$ is the tree $\pi(\Lambda_f)$.   
\end{enumerate}
\end{prop}

To prove Theorem \ref{TeoSemi}, for each measure $\mu \in \M _a(f)$, we construct a tree $\X_{f,\mu}$ and an induced dynamics $\hat f_\mu$ such that there exists a continuous projection $\pi_\mu:\D\to \X_{f,\mu}$ verifiying $ \pi_\mu \circ  f = \hat f_\mu \circ \pi_\mu$.   We call $(\X_{f,\mu},\hat f_\mu)$ the natural tree model associated to $(f,\mu)$  (see section \ref{treeConstruction} for more details). The finite ergodic covering property means that a finite number of these trees contain all the dynamical information of $(\X_f,\hat f)$. Then, theorem  \ref{TeoSemi} is a consequence of the following result. 

\begin{teoprin}\label{TeoSemimu}
Let $f,g$ in  $\text{MD}^r (\D)$ be stable efficient, $\Lambda_f$ and $\Lambda_g$ be the maximal invariant sets of $f$ and $g$ respectively in $\D$. Suppose there exists a homeomorphism $h:\Lambda_f\to \Lambda_g$ such that $h\circ f_{|_{\Lambda_f}} = g_{|_{\Lambda_g}} \circ h$. Consider $\mu \in \M_a(f)$, $\nu=h_*(\mu)$,  $(\X_{f,\mu},\hat f_\mu)$ and $(\X_{g,\nu}, \hat g_\nu)$ the natural tree models associated to $(f,\mu)$ and $(g,\nu)$. Then, there exists two continuous surjective maps $\hat h_1:\X_{f,\mu}\rightarrow \X_{g,\nu}$ and $\hat h_2:\X_{g,\nu}\rightarrow \X_{f,\mu}$ such that $\hat h_1\circ\hat{f}_\mu=\hat{g}_\nu\circ \hat h_1$ and  $\hat h_2\circ\hat{g}_\nu=\hat{f}_\mu\circ \hat h_2$.  	
\end{teoprin}

Having a negative answer to Question \ref{QuestTreeUni} and in light of theorem \ref{TeoSemi}, it is natural to wonder the following:
\begin{quest}\label{questPrin}
Given $f\in \text{MD}^r (\D)$ stable efficient, are all possible tree models semi-conjugate to each other in both directions like in Theorem \ref{TeoSemi}? 
\end{quest}

A positive answer to this question implies that mildly dissipative maps are in essence a two-dimensional differentiable model of ``conjugacy" classes of one-dimensional endomorphisms in trees. 

A step toward getting a positive answer to Question \ref{questPrin} is to understand the converse of Question \ref{QuestInv}.
	
\begin{quest}
Given $f,g\in \text{MD}^r (\D)$, consider $(\X_f,\hat{f})$ and $(\X_g,\hat{g})$ associated to $f$ and $g$ respectively. If the dynamics in $\hat f$ and $\hat g$ are equivalent, then are the dynamics of $f$ and $g$ also equivalent? 
\end{quest}

The answer to this question depends on the type of equivalency we are looking for. For the ergodic type, we show a positive result.

\begin{teoprin}\label{TeoConvErg}
Let $f,g\in \text{MD}^r (\D)$ be stable efficient, and suppose there exist $Y_f\subset \X_f$, $Y_g\subset \X_g$, and $\hat h: Y_f\to Y_g$ -- a measurable bijection between sets of full measure for every aperiodic ergodic measure of $\hat{f}$ and $\hat{g}$. If $\hat{g}\circ \hat h= \hat h\circ \hat{f}_{|Y_f}$ then, there exist $M_f\subset \Lambda_f$ and $M_g\subset \Lambda_g$ and $h:M_f\to M_g$ such that:
\begin{enumerate}
\item $M_f$ and $M_g$ have full measure for every aperiodic ergodic measure of $f$ and $g$, respectively. 
\item $\pi_{g|_{M_g}} \circ h = \hat h \circ \pi_{f|_{M_f}}$
\item $g_{|_{M_g}}\circ h = h \circ f_{|_{M_f}}$ 
\item The  map $h_*:\M_a(f)\to \M_a(g)$ induced by $h$ is a bijection.  
\end{enumerate}
\end{teoprin}

Despite $\pi_*$ being a bijection between aperiodic ergodic measures, the previous theorem is not immediate, because there is no canonical way to define the inverse of $\pi_f:\D\to \X_f$. The conjugacy $\hat h$ gives us a bijection between segments of stable manifolds, yet it does not help identify which point goes into which point. However, by the properties of the stable manifolds, we can extract this information through the inverse limits of the trees (see \cite{Viana} for an exposition on inverse limits).  In \cite{borosti}, for Wang--Young parameters in the Henon family and for the Misiurewicz--\v{S}timac parameters of Lozi maps, the authors show that the inverse limits of the tree are conjugate to the maximal invariant set of $f$.

This paper is structured as follows:
\begin{itemize}
\item In section \ref{treeConstruction}, we do a brief recall on how natural tree models are built. 
\item In section \ref{CSH}, we analyze possible ways to construct a tree model associated to the horseshoe and in particular show and explain Example \ref{contra}.
\item In section \ref{SecProp}, we prove Propositions \ref{PropProp} and \ref{PropStabEff}.
\item In section \ref{SecTeoSemi}, we prove Proposition \ref{ConjDisco} and Theorems \ref{TeoSemi} and \ref{TeoSemimu}.
\item In section \ref{SecConv}, we prove Theorem  \ref{TeoConvErg}.
\end{itemize}

\section{Construction of the tree associated to a mildly dissipative map}\label{treeConstruction}
In this section, we briefly recall the construction carried out in \cite{pujcrov}.

\subsection{Regular points and hyperbolic blocks}\label{subSecRegHypBlo}

Let us recall the basics of Pesin theory. In our context, a hyperbolic block is a compact set $B$ (not necessarily invariant) such that there exist $0<\lambda< 1$, $C>0$, and for each $x\in B$, there are two families of one-dimensional subspaces $E^s(f^n(x)), E(f^n(x)) \subset \R^2$ verifying the following:
\begin{enumerate}
\item $\R^2= E^s(f^n(x))\oplus E(f^n(x))$,  $df^n_x(E^s(x))= E^s(f^{n}(x))$, $df^n_x(E(x))= E(f^{n}(x))$.
\item $|df^n_x(v)|\leq C\lambda^n |v|$, for all $n> 0$ and for all $v\in E^s(x)$
\item $|df^n_x(v)|\geq C^{-1} \lambda^{n} |v|$, for all $n< 0$ and for all $v\in E^s(x)$ 
\item $\lim_{n\to \pm \infty}\frac{1}{n}log(|df^n_x(v)|)\geq 0$ for all $v\in E(x)$ 
\end{enumerate} 

For $f\in\text{MD}^r (\D)$ and an ergodic aperiodic measure $\mu$, dissipativness of $f$, aperiodicity of $\mu$,  and Oseledets theorem implies the existence of hyperbolic blocks with positive measure. Moreover, the union of every block ($\cup_{\lambda, C}B$) has full measure. 

We may consider a countable family of blocks $\{B_n\}_{n\in\N}$ ordered by inclusion such that $\mu(\cup_n B_n)=1$. We say that a point $x\in \D$ is regular if its orbit is dense in the support of the measure, belongs to a block $B_n$, and the intersection between the orbit of $x$ and $B_m$ is dense in $B_m$ for every $m\geq n$. We use $B^*_n$ to represent the set of regular points in $B_n$ and $R_{f,\mu}$ to represent the set of all regular points associated to $\mu$. Based on the ergodicity of $\mu$ and the Poincaré recurrence theorem, $\mu(B^*_n)=\mu(B_n)$ and therefore $\mu(R_{f,\mu})=1$. The extra care we took previously allows us when considering two regular points $x,y$ that they belong to the same block. From now on, when considering blocks, we mean one of the family $\{B_n\}$, but we shall lose the $n$ in the notation. 

For regular points, we have a stable manifold theorem in each block. We shall call $ \operatorname{Emb^1}((-\hat \varepsilon,\hat\varepsilon),\D^2)$ the space of $C^1$ embeddings of the interval $(-\hat \varepsilon,\hat\varepsilon)$ into $\D^2$ with $C^1$ topology. 

\begin{teo}[Stable manifold]\label{teoStaMan}
Consider $f\in\text{MD}^r (\D)$, $\mu$  an ergodic aperiodic measure, and $B$ a hyperbolic block for $\mu$. Then, there exist $\hat \varepsilon$ and a continuous map $\psi:B^*\to \operatorname{Emb^1}((-\hat \varepsilon,\hat\varepsilon),\D^2)$ verifying the following:
\begin{enumerate}
\item $\psi(x)((-\varepsilon,\varepsilon))=W^s_\varepsilon(x) =\{y\in\D^2:d(f^n(x),f^n(y))<\varepsilon,\ \forall n\geq 0\}$ for all $\varepsilon \leq \hat\varepsilon$ and all $x\in B^*$ 
\item $T_x W^s_\varepsilon (x) = E^s(x)$ for all $x\in B^*$ 
\item For some constant $C(\varepsilon)$, holds $length(f^n(W^s_\varepsilon(x)))\leq C\lambda^n$, for all $n>0$ and $x\in B^*$ 
\item  If $\{n_i\}_{i\in \N}=\{n\in\N: f^{n}(x)\in B^*\}$, then 
\[W^s(x)=\{y\in M:lim_n\ d(f^n(y),f^n(x))=0\}=\cup_{i\in \N}f^{-n_i}(W^s_\varepsilon(f^{n_i}(x))).\]
\end{enumerate} 
\end{teo}

See chapter 7 in \cite{Pesin} for a proof of this theorem. 

We now define for each regular point $x$ of $\mu$ a collection of curves $\Gamma_{\mu,x}$. This collection verifies the following:
\begin{enumerate}
\item[P1] Each $\gamma\in\Gamma_{\mu,x}$ is a connected component of  $W^s(f^n(x))\cap \D$ for some $n\in\Z$. 
\item[P2] Each segment of each $\gamma\in\Gamma_{\mu,x}$ is the $C^1$ limit of segments of curves in $\Gamma_{\mu,x}$.  Moreover, this accumulation happens from both sides. 
\item[P3]  For each $\gamma\in\Gamma_{\mu,x}$, the connected components of $f^{-1}(\gamma)\cap \D$ that intersect $f(\D)$  also belong to $\Gamma_{\mu,x}$.
\item[P4] For any other $y\in R_{f,\mu}$, every segment of every curve in $\Gamma_{\mu,y}$ is the $C^1$-limit of segments of curves in $\Gamma_{\mu,x}$. Also, this accumulation happens from both sides. 
\end{enumerate}	
	
We would like to make two observations on property P1. First, $f^n(x)$ may not belong to $\gamma$, and, second, the curves in $\Gamma_{\mu,x}$ either coincide or are pairwise disjoint.
	
To obtain such a family, we choose $x\in R_{f,\mu}$, a regular point for $\mu$, and consider $\Gamma_0$ the collection of connected components of $W^s(f^{n}(x))\cap \D$ that contains $f^{n}(x)$, with $n\in \Z$. Now,  we define $\Gamma_i$ by induction. For $i>0$,  $\Gamma_i$ is the collection of connected components of $f^{-1}(\gamma)\cap \D$ that intersects $f(\D)$ for $\gamma\in\Gamma_{i-1}$. Once we have $\Gamma_i$ defined, we consider $\Gamma_{\mu,x}=\bigcup_{i\geq 0}\Gamma_{i}$.

 By construction, $\Gamma_{\mu,x}$ naturally verifies P1 and P3. Since $x$ is regular, its orbit is dense in the support of $\mu$, which implies accumulation in properties P2 and P4 for at least one side. The accumulation by both sides is granted because $\Gamma_\mu$ is countable. If the accumulation happens from only one side, then the support of $ \mu $ would be countable. This is a contradiction with  $\mu$ being aperiodic.

\subsection{The tree model}\label{subsectreemodel}

We shall construct $\X$ and $\hat f:\X\to \X$  the tree model of $(f,\D)$. We consider
\[\Gamma=\bigcup_{\mu\in \M_a(f)} \Gamma_{\mu,x},\]
where   $x$ is a single point in $R_{f,\mu}$. This collection  $\Gamma$ verifies P1, P2, and P3, as well as:
\begin{itemize}
	\item[P4*] 
	Given any aperiodic ergodic measure $ \mu $, there is a full measure set of points $y$ such that the connected components of $ W^s (y) \cap \D $ are $ C^1 $- limit of arcs in $\Gamma$ and are accumulated by both sides.
	\end{itemize}

The property P4 of $\Gamma_{\mu,x}$ implies that this construction does not depend on the choice of $x$. Now that we have the collection of curves $\Gamma$, we proceed to define $\X$ as follows. We use $s$ to denote any compact connected surfaces $ s\subset\D $ whose boundary consist of a finite number of curves from $\Gamma$ and  $\partial\D$. We define the collection of nested sequences of such surfaces as 
\[\Sigma=\{(s_n)_{n\in\N}: \overline{s_{n+1}}\subset \interior(s_n) , \text{ for each } n \}.\]
Given $(s_n)$ and $(s'_n)$ in $\Sigma$, we say that $(s_n)\leq (s'_n)$ if for all $n$, there exists $m$ such that $\overline{s_m}\subset s'_n$. We say that $(s_n)\sim (s'_n)$ if $(s_n)\leq (s'_n)$ and $(s'_n)\leq (s_n)$. Since $\sim$ is an equivalence relation, we define $\tilde \Sigma = \Sigma/\sim$. By our definition of $\leq$, it induces a partial order in $\tilde \Sigma$, and with it, we define $\X$ as the set of minimal elements of $\tilde \Sigma$.

Property P2 implies that each element of $\gamma\in \Gamma$ is also an element of $\X$. Since $\X$ is also a partition of $\D$, the projection $\pi:\D\to \X$ is naturally defined. We shall represent the elements of $\X$  by $\hat x$ in general or by $\gamma$ if it is a curve of $\Gamma$. We consider in $\X$ the quotient topology, and, by definition, $\pi$ is a continuous function for this topology. Therefore, $\X$ is compact, and it is easily identifiable as a Hausdorff space as well. Finally, from the property that defines mildly dissipative maps, it is inferred that $\X$ is a real tree. By this, we mean that for every pair of points $x,y\in \X$, there exists a unique (up to reparametrizations) curve $\gamma:[0,1]\to \X$ that is continuous and injective with $\gamma(0)=x$ and $\gamma(1)=y$. 

We say that $f$ has a finite ergodic covering if there exists a finite number of measures $\mu_1,\cdots, \mu_l\in \M_a(f)$ such that $\pi(\cup_{i=1}^{l} \Gamma_{\mu_i,x_i})$ is dense in $\X$.  

The authors in \cite{pujcrov} built the natural tree model $(\X,\hat f)$ upon every ergodic aperiodic measure of $f$. However, if we fix $\mu$, the same processes using $\Gamma_{\mu,x}$ instead of $\Gamma$ produce a tree $\X_\mu$, a map $\hat f_\mu: \X_\mu\to \X_\mu$, and a continuous projection $\pi_\mu :\D\to \X_\mu$ verifying $\pi_\mu \circ f = \hat f_\mu \circ \pi_\mu$.

When dealing with two maps $f$ and $g$, to differentiate the trees associated to each map, we shall write $\X_f$, $\X_g$ for the general trees and $\X_{f,\mu}$ and $\X_{g,\nu}$ for the trees associated to single measures $\mu$ and $\nu$.

\section{Case study: the horseshoe}\label{CSH}

In this section, we study the tree associated to the classical construction of the horseshoe in a disc show how to cut regions in order to obtain a different tree as claimed in Example \ref{contra}.

Let $f:\D\rightarrow f(\D)$ be the classical construction of the horseshoe in a disc. In this case, the points that have a stable manifold to consider are those in the horseshoe. Here, the family of stable manifolds $\{W^s_\D(x)\}$ is totally ordered, and therefore $\X$ is an interval (see Figure \ref{TreeHorseshoe01}). Moreover, it is simple to observe that $\hat{f}$  is in fact the tent map. 

\begin{figure}[h]
\centering
\input{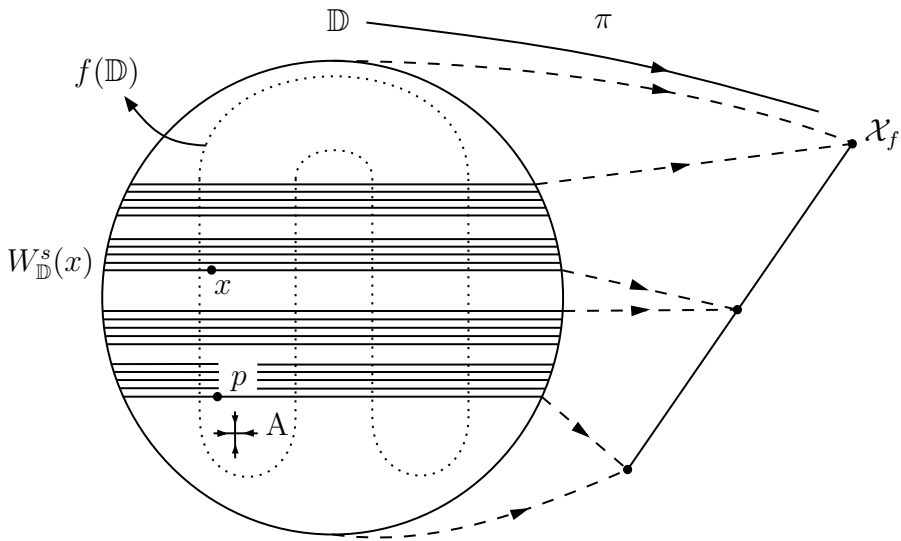}
\caption{Tree model associated to the horseshoe in $\D$.}\label{TreeHorseshoe01}
\end{figure}

We would like to point out that in this case, $\pi_{*}$ is not injective in the whole set $\mathcal{M}_1(f)$. Note that $A$, the attracting fixed point of $f$, and $p$, one of the two hyperbolic fixed points of $f$, are projected on the same fixed point in $\X$. Therefore, the Dirac measures $\delta_A$ and $\delta_p$ have the same image under  $\pi_{*}$.

We proceed to explain Example \ref{contra}. We need two maps $f$ and $g$, and we take $(f,\D)$ as the classical horseshoe we had worked with. To construct $g$, we cut a region from $\D$, obtaining a set $D$ that is also a disc. We then take a diffeomorphism $h:\D\to D$ and define $g=h^{-1}\circ f_{|D}\circ h$. For the map $g$ to be well defined, we only need that $f(\D)\subset D$.

Consider $D$ as in Figure \ref{TreeHorseshoe02}, and observe that half of the stable manifolds $W^s_\D(x,f)$ have been split in two and the other half not. In particular, the family of curves $W^s_\D(x,g)$ is not  totally ordered anymore, and $\X_g$ is thus not an interval. In essence, we have cut through half of the interval, obtaining a tree with three branches.

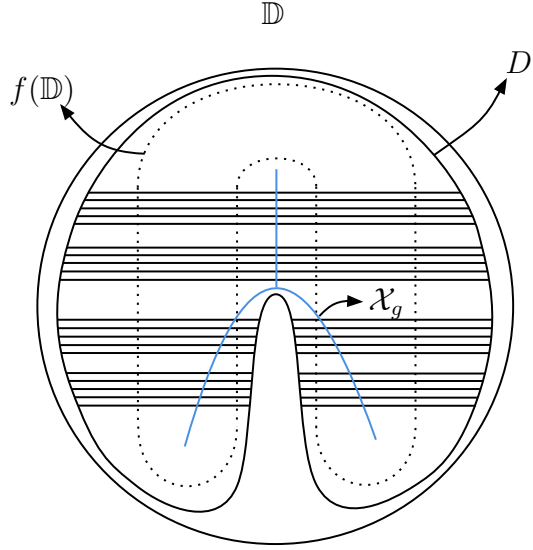
\begin{figure}[h]
\centering
\tikzset{every picture/.style={line width=0.75pt}} 

\begin{tikzpicture}[x=0.75pt,y=0.75pt,yscale=-1,xscale=1]

\draw   (119.45,159.97) .. controls (119.45,93.71) and (173.16,40) .. (239.43,40) .. controls (305.69,40) and (359.4,93.71) .. (359.4,159.97) .. controls (359.4,226.24) and (305.69,279.95) .. (239.43,279.95) .. controls (173.16,279.95) and (119.45,226.24) .. (119.45,159.97) -- cycle ;
\draw    (145.27,102.59) -- (333,102.59) ;
\draw  [dash pattern={on 0.84pt off 2.51pt}]  (170.08,100.04) .. controls (169.92,30.54) and (310.13,30.44) .. (310.08,99.79) ;
\draw  [dash pattern={on 0.84pt off 2.51pt}]  (220.17,99.88) .. controls (220.25,80.31) and (260.25,80.56) .. (260.08,99.54) ;
\draw  [dash pattern={on 0.84pt off 2.51pt}]  (170.08,100.04) -- (170.17,219.75) ;
\draw  [dash pattern={on 0.84pt off 2.51pt}]  (220.17,99.88) -- (220.08,220.54) ;
\draw  [dash pattern={on 0.84pt off 2.51pt}]  (260.08,99.54) -- (260.08,220.38) ;
\draw  [dash pattern={on 0.84pt off 2.51pt}]  (310.08,99.79) -- (310.17,220.04) ;
\draw  [dash pattern={on 0.84pt off 2.51pt}]  (260.08,220.38) .. controls (260,259.9) and (310.13,260.27) .. (310.17,220.04) ;
\draw  [dash pattern={on 0.84pt off 2.51pt}]  (170.33,220.63) .. controls (170.17,260.9) and (220.17,260.75) .. (220.08,220.54) ;
\draw    (131.75,58.94) .. controls (144.38,74.5) and (150.75,82.06) .. (173.63,82.94) ;
\draw  [fill={rgb, 255:red, 0; green, 0; blue, 0 }  ,fill opacity=1 ] (131.75,58.94) -- (137.37,62.43) -- (133.66,65.28) -- cycle ;
\draw    (252.73,210.14) -- (340.47,210.14) ;
\draw    (252.33,206.05) -- (341.67,206.05) ;
\draw    (251.8,201.71) -- (343,201.71) ;
\draw    (143.13,106.34) -- (335.27,106.34) ;
\draw    (251.4,197.65) -- (344.47,197.65) ;
\draw    (250.33,183.57) -- (347.53,183.57) ;
\draw    (249.27,175.25) -- (348.33,175.25) ;
\draw    (249.8,179.45) -- (348.2,179.45) ;
\draw    (135.4,197.62) -- (227.4,197.62) ;
\draw    (138.2,205.85) -- (226.6,205.85) ;
\draw    (136.5,201.71) -- (227.27,201.71) ;
\draw    (139.4,114.46) -- (339.27,114.46) ;
\draw    (141.27,110.52) -- (337.4,110.52) ;
\draw    (134.33,130.33) -- (343.93,130.33) ;
\draw    (132.33,138.07) -- (345.4,138.07) ;
\draw    (133.27,134.16) -- (344.47,134.16) ;
\draw    (131.4,142.36) -- (346.33,142.36) ;
\draw    (248.2,170.99) -- (348.73,170.99) ;
\draw    (134.07,193.64) -- (227.67,193.64) ;
\draw    (251,193.88) -- (345.4,193.88) ;
\draw    (247.53,166.73) -- (348.73,166.73) ;
\draw    (138.2,118.34) -- (340.6,118.34) ;
\draw    (130.87,146.61) -- (347.27,146.61) ;
\draw   (137.5,119.88) .. controls (144.75,97.88) and (174.75,43.63) .. (237.75,43.63) .. controls (300.75,43.62) and (336,100.38) .. (341.75,122.38) .. controls (347.5,144.38) and (353.75,167.13) .. (343,201.63) .. controls (332.25,236.13) and (319.5,244.13) .. (309.5,250.63) .. controls (299.5,257.13) and (272.5,272.38) .. (260.5,250.63) .. controls (248.5,228.88) and (254.5,153.88) .. (239.75,153.88) .. controls (225,153.88) and (229.5,227.63) .. (220.25,250.38) .. controls (211,273.13) and (184,262.13) .. (169.5,250.63) .. controls (155,239.13) and (148.75,234.38) .. (136.5,201.13) .. controls (124.25,167.88) and (130.25,141.88) .. (137.5,119.88) -- cycle ;
\draw    (319.5,83.63) .. controls (339.25,69.63) and (345.25,66.63) .. (353.25,50.13) ;
\draw  [fill={rgb, 255:red, 0; green, 0; blue, 0 }  ,fill opacity=1 ] (355.6,45.92) -- (354.72,52.48) -- (350.6,50.26) -- cycle ;
\draw    (231.4,166.73) -- (129.53,166.73) ;
\draw    (230.47,171) -- (129.8,171) ;
\draw    (229.67,175.27) -- (130.33,175.27) ;
\draw    (229.4,179.53) -- (131,179.53) ;
\draw    (228.73,183.53) -- (131.67,183.53) ;
\draw    (139.93,210.07) -- (226.47,210.07) ;
\draw [color={rgb, 255:red, 74; green, 144; blue, 226 }  ,draw opacity=1 ]   (193.8,230.7) .. controls (223,127.5) and (253.2,122.1) .. (290.2,227.3) ;
\draw [color={rgb, 255:red, 74; green, 144; blue, 226 }  ,draw opacity=1 ]   (240,90.75) -- (240,150.3) ;
\draw    (261.65,164.9) .. controls (267.2,159.1) and (270.6,156.7) .. (280.4,157.3) ;
\draw  [fill={rgb, 255:red, 0; green, 0; blue, 0 }  ,fill opacity=1 ] (282.39,157.62) -- (276.09,159.66) -- (276.31,154.99) -- cycle ;

\draw (230.99,4.76) node [anchor=north west][inner sep=0.75pt]    {$\mathbb{D}$};
\draw (103.8,42) node [anchor=north west][inner sep=0.75pt]    {$f(\mathbb{D})$};
\draw (354,30) node [anchor=north west][inner sep=0.75pt]    {$D$};
\draw (284.92,149.32) node [anchor=north west][inner sep=0.75pt]    {$\mathcal{X}_{g}$};

\end{tikzpicture}
\caption{Associated tree of the horseshoe in $D$}\label{TreeHorseshoe02}
\end{figure}

\section{Proof of Propositions \ref{PropProp} and \ref{PropStabEff}}\label{SecProp}

We begin this section with the proof of Proposition \ref{PropProp}

\begin{proof}[of Proposition \ref{PropProp}]

Given $f\in \text{MD}^r(\D)$ and its natural induced tree model $(\hat{f},\X)$, consider $\Gamma$ the family of curves constructed in section \ref{treeConstruction}.

Let us begin by proving that 	$\pi(\per(f))= \per(\hat{f})$. It is clear that $\pi(\per(f))\subset \per(\hat f)$, and therefore we only need to show the remaining inclusion. Consider a periodic point $\hat p\in \per(\hat{f})$ of period $k$. First,  observe that $\hat p$ is not represented by any element in $\Gamma$, since elements in $\Gamma$ are associated to ergodic aperiodic measures. We consider the compact set
\[K = \pi^{-1}(\hat p \cup \hat f (\hat p) \cup \cdots \cup \hat f^{k-1}(\hat p))\subset \D \]
and observe that $f(K)\subset K$. Therefore, we conclude the existence of an ergodic measure $\mu$ such that $support(\mu)\subset K$. If this measure were aperiodic, there would be an infinite amount of points inside of $K$ that project in $\X$ injectively. However, this cannot happen because $\pi(K)$ is the orbit of $\hat p$, a finite set. Then, $\mu$ is periodic, and with this, we infer the existence of $p\in \D$ such that $\pi(p)=\hat p$.

We proceed now to prove that $\pi_*: \M_1(f)\to \M_1(\hat f)$ is surjective. Since $\pi_*$ is a convex map and $\pi(\per(f))= \per(\hat{f})$, we only need to prove that $\pi_*$ is surjective in the set of aperiodic ergodic measures. Consider $\hat{\mu}$ an aperiodic ergodic measure of $\hat f$ and $\hat{x}\in \X$ such that 
\[\lim_{n\to\infty}\frac{1}{n}\sum_{i=0}^{n-1}\delta_{\hat{f}^i(\hat{x})}=\hat{\mu}.\]
From  surjectivity of $\pi$, there exists $x\in \D$ such that $\pi(x)=\hat{x}$. Since $f$ is continuous, the sequence of measures  $\frac{1}{n}\sum_{j=0}^{n-1}\delta_{f^j(x)}$ has at least one accumulating point $\mu\in \M_1(f)$.  Consider a sequence $n_k$ such that $\lim_{k\to\infty}\frac{1}{n_k}\sum_{i=0}^{n_k-1}\delta_{f^i(x)}=\mu$ and observe that
\begin{equation}
\begin{split}
\pi_* (\mu) & = \pi_* \left(\lim_{k\to\infty}\frac{1}{n_k}\sum_{i=0}^{n_k-1}\delta_{f^i(x)}\right) \\
& = \lim_{k\to\infty} \frac{1}{n_k}\sum_{i=0}^{n_k-1}\pi_*(\delta_{f^i(x)}) \\
& = \lim_{k\to\infty} \frac{1}{n_k}\sum_{i=0}^{n_k-1}\delta_{\hat{f}^i(\hat{x})} \\
& = \hat \mu.
\end{split}
\end{equation}
 This proves our assertion.
\end{proof}

We now aim to prove Proposition \ref{PropStabEff}. Let us begin with an equivalent definition for stable efficiency. Given $f\in \text{MD}^r(\D)$, consider $\Gamma$ the family of curves constructed in section \ref{treeConstruction} and $\Lambda_f$ the maximal invariant set of $f$ in $\D$. We say that $\gamma\in \Gamma$ splits $\Lambda_f$ if at least two components of $\D\setminus \gamma$ intersect $\Lambda_f$.

\begin{lema}
 Given $f\in \text{MD}^r(\D)$, $f$ is stable efficient if and only if every curve $\gamma\in \Gamma$ splits $\Lambda_f$. 
\end{lema}

This lemma is the reason why we call the property stable efficiency. During the inductive construction of  $\Gamma_i$, if we only get curves $\gamma$ that split the maximal invariant set, then $\pi(\Lambda_f) = \X$.   

\begin{proof}
$(\Longrightarrow):$ Let us assume that $\pi(\Lambda_f) = \X$, and consider $\gamma\in \Gamma$. Property P2 in the construction of $\Gamma$ implies that $\gamma$ is never an endpoint of  $\X$.  Thus, $\X \setminus \{\gamma\}$ has at least two connected components. In each of these components, we have points of $\pi(\Lambda_f)$, and in at least two components of $\D\setminus \gamma$, we have points of $\Lambda_f$.  

$(\Longleftarrow):$ Suppose that every curve $\gamma\in \Gamma$ splits $\Lambda_f$; then we claim that $\gamma \cap \Lambda_f \neq \emptyset$ for every $\gamma \in \Gamma$. It is easy to observe that if $C$ is a connected set that intersects $A$ and $A^c$, then it also intersects its border. We use this assertion considering $C= \Lambda_f$, $A$ some connected component of  $\D\setminus \gamma$ that intersects $\Lambda_f$, and the fact that the border of $A$ is a subset of $\gamma$. Therefore, $\Gamma\subset \pi(\Lambda_f)$. Since $\Gamma$ is dense in $\X$ and  $\pi(\Lambda_f)$ is closed, we conclude $\pi(\Lambda_f)=\X$.
\end{proof}

The previous lemma gives the idea of for proving  Proposition \ref{PropStabEff}: remove from the disk the regions that contain curves $\Gamma$ that do not split $\Lambda_f$. 

\begin{proof}[of proposition \ref{PropStabEff}]
Consider $\Gamma^+=\{\gamma \in \Gamma:\ \gamma\text{ splits}\ \Lambda_f\}$ and $\Gamma^- = \Gamma \setminus \Gamma^+$. Let $\Delta^-$ be the family of compact connected surfaces $S^-$ whose border (in $\D$) are curves in $\Gamma^-$ and  such that $\Lambda_f\subset S^-$.  

Define $D^- = \bigcap_{S^-\in \Delta^-} S^-$ and $D=\pi^{-1}(\pi(\Lambda_f))$. By our previous lemma and the continuity of $\pi$, we infer that $ D = D^-$. From this equality, it is clear that $D$ is also a disk.  

Now, observe that $f_{|D}$ remains mildly dissipative. The association $\mu\in \M_a(f)\mapsto \mu_{|D}\in M_a(f_{|D})$ defined by $\mu_{|D}(A)=\mu(A\cap D)$ is a bijection. Moreover, regular points remain regular; $R_{f,\mu}= R_{f_{|D},\mu_{|D} }$. Consider $\tilde \Lambda_f$ the maximal invariant set of $f_{|D}$. We already know that $\tilde \Lambda_f\subset \Lambda_f$ because $D\subset \D$. By definition of $D^-$, we also conclude the other inclusion; therefore,  $\tilde \Lambda_f= \Lambda_f$. 

For each $x\in R_{f_{|D},\mu_{|D} }$, we shall call  $\tilde \Gamma_{\mu,x}$ the family of curves $\Gamma_{\mu_{|D},x}$, as in subsection \ref{subSecRegHypBlo} for $(f_{|_D},D)$ instead of $(f,\D)$. Then, we define  $\tilde \Gamma= \bigcup_\mu \tilde \Gamma_{\mu,x}$ and claim that $\tilde \Gamma = \Gamma^+$.  The inclusion $\tilde \Gamma\subset \Gamma^+$ happens because $\tilde \Gamma \subset \{\gamma\in \Gamma: \gamma \subset D\} = \Gamma^+$. On the other hand, the curves in $\Gamma^+$ are not eliminated in the inductive process. Since the curves in $\Gamma^+$ verify $\gamma\cap \Lambda_f \neq \emptyset$ and $\Lambda_f = \tilde\Lambda_f$, they also verify $\gamma \cap f(D)\neq \emptyset$. Therefore, these curves remain, and by the previous lemma, we deduce that $f_{|_D}$ is stable efficient. In particular, the tree associated to $f_{|_D}$ is the tree $\pi(\Lambda_f)$. 
\end{proof}

\section{Proof of Theorems \ref{TeoSemi} and \ref{TeoSemimu}}  \label{SecTeoSemi}

In this section, we prove Theorems \ref{TeoSemi} and \ref{TeoSemimu}. We now proceed to prove Proposition \ref{ConjDisco}, which will be used to prove Theorem \ref{TeoSemimu}. 

\begin{proof}[of Proposition \ref{ConjDisco}]

Consider $f, g \in \text{MD}^r(\D)$, and suppose there exists a homeomorphism $h:\D\to \D$ such that $h\circ f=g\circ h$. Our aim is to construct a homeomorphism $\hat h: \X_f\to \X_g$ that conjugates $\hat{f}$ and $\hat{g}$.
 
Given $\mu\in \M_a(f)$, we define  $\nu = h_*\mu\in \M_a(g)$. Thus, $h$ induces a bijection between the ergodic aperiodic measures of $f$ and $g$. Consider $R_{f,\mu}$ the set of regular points of $\mu$ and $R_{g,\nu}$ the set of regular points of $\nu$. Since $\mu(R_{f,\mu}) = 1$ and $\nu(R_{g,\nu})=1$, we can assume that $h( R_{f,\mu})= R_{g,\nu}$. If this is not the case, then we can take the intersection, and the measure of this set must be 1. Now, consider the collections of curves $\Gamma_{\mu,x}$ and $\Gamma_{\nu,h(x)}$ for some $x\in R_{f,\mu}$, as defined in subsection \ref{subSecRegHypBlo}. Recall that we defined $\Gamma_{\mu,x}$ as the union of a family of curves $\Gamma_i$ with $i\in \N$. To distinguish the families of curves $\Gamma_i$ associated to $\mu$ from the families of curves $\Gamma_i$ associated to $\nu$, we shall call them $\Gamma_{\mu,i}$ and $\Gamma_{\nu,i}$ respectively. Since $h$ is defined in the whole disk, $h(W^s_\D(x))= W^s_\D(h(x))$.  Then, $h$ provides a map  $\hat h:\Gamma_{\mu,0}\to\Gamma_{\nu,0}$  defined by $\hat h(\gamma)=h(\gamma)$ that naturally extends first to a bijection between $\Gamma_{\mu,x}$  and  $\Gamma_{\nu,h(x)}$ and later to a bijection between $\Gamma_f$ and $\Gamma_g$. To illustrate this, pick $\hat \gamma \in \Gamma_{\mu,0}$ and consider the subsets $A= \{\gamma\in \Gamma_{\mu,1}: f(\gamma)=\hat \gamma\}$ and $B=\{\gamma\in \Gamma_{\nu,1}: g(\gamma)=\hat h(\hat \gamma)\}$. By the conjugacy  $h$, we know  $\gamma \in A\mapsto h(\gamma)\in B$ is a bijection; thus, $\hat h (\gamma) = h(\gamma)$ is well defined from $\Gamma_{\mu,1}$ to $\Gamma_{\nu,1}$. 

Consider the projections $\pi_f:\D\rightarrow \X_f$ and $\pi_g:\D\rightarrow \X_g$ and define the map $\hat h:\X_f\rightarrow \X_g$ as $\hat h(\pi_f(x))=\pi_g(h(x))$ for  each $x\in\D$. We claim that $\hat h$ is well defined. Based on our previous discussion, if $\pi_f(x)$ represents a curve in $\Gamma_f$, then it is clear. On the other hand, consider $y,z\in\D$, such that $\pi_f(y)=\pi_f(z)=(s_n)$, where $(s_n)$ is a nested sequence of surfaces in $\Sigma$. We know by definition that  $y, z\in \bigcap s_n$. Suppose by contradiction that $h(y)$ and $h(z)$ are separated  by some curve $\gamma\in\Gamma_g$, since $h^{-1}$ is  homeomorphism $h^{-1}(\gamma)$  separates $y$ and $z$, which is an absurd. 

Once $\hat h$ is well defined, its continuity is deduced by the continuity of $\pi_f$, $\pi_g$, and $h$. We can also see that
\begin{equation}
\begin{split}
\hat h ( \hat f (\pi_f(y)))& = \hat h (\pi_f(f(y))) = \pi_g(h(f(y)))= \pi_g(g(h(y))) \\
& = \hat g( \pi_g(h(y))) = \hat g (\hat h (\pi_f(y))),
\end{split}
\end{equation}
and since $\pi_f$ is surjective, we infer $\hat h \circ \hat f  = \hat g \circ \hat h$.

Analogously, with  $h^{-1}$, we define  the inverse of $\hat h$, which verifies
$\hat h^{-1}\circ\pi_g=\pi_f\circ h^{-1}$ and then $\hat h^{-1}\circ \hat{g}=\hat{f}\circ \hat h^{-1}$.

\end{proof}

We also deduce that the same result holds for the trees $\X_\mu$ associated to a single measure $\mu$. 

\begin{lema}\label{ConjDiscoMed}
	Let $f$  and $g$ be in  $ \text{MD}^r (\D)$, and suppose there exists a homeomorphism $h:\D\to \D$  that verifies $g\circ h = h \circ f$. Given $\mu \in \M_a(f)$ and $\nu = h_*(\mu)\in \M_a(g)$,  there exists a homeomorphism $\hat h_{\mu,\nu}:\X_{f,\mu}\to \X_{g,\nu}$  such that $\hat g_\nu \circ \hat h_{\mu,\nu} = \hat h_{\mu,\nu} \circ \hat f_\mu$. 
\end{lema}

Let us consider $f\in \text{MD}^r(\D)$ and observe that $f_{|f(\D)}\in \text{MD}^r(f(\D))$. Now, we can construct a tree model for $(f,\D)$ and another one for $(f_{|f(\D)}, f(\D))$. Let us call them $(\X,\hat f)$ and $(\X^1,\hat f_1)$ respectively. Since $f:\D\to f(\D)$ is a homeomorphism that conjugates $f$ and $f_{|\D}$ by our previous proposition, we note that $(\X,\hat f)$ and $(\X^1,\hat f_1)$ are conjugate. In particular, the following lemma is true.

\begin{lema}\label{LemTreeK}
Given $f\in \text{MD}^r(\D)$ and $k\geq 1$, if $(\X^k,\hat f_k)$ is the natural tree model associated to $f_{|f^k(\D)}\in \text{MD}^r(f^k(\D))$, then $( \X,\hat f)$ is conjugate (by a homeomorphism) to $(\X^k,\hat f_k)$. Moreover, for  any $\mu \in \M_a(f)$, $(\X_\mu, \hat f_\mu)$ is conjugate (by a homeomorphism) to $(\X^k_\mu, \hat f_{k,\mu})$. 
\end{lema}

To give a sense of the proof of Theorem \ref{TeoSemi}, let us explain why there are semi-conjugacies in both directions in  Example \ref{contra} (see Figure \ref{ArvDinConMaxInv}). Recall that in Example \ref{contra},  $(f,\D)$ is the classical definition of the horseshoe in the disk, and $g$ is (up to a conjugation) the restriction of $f$ to $D$ that is homeomorphic to a disk. In other words, we may consider $g=f_{|D}$. It is easily to observe that any $\gamma_g\in \Gamma_g$ is contained in some $\gamma_f\in \Gamma_f$. This induces the inclusion map $i_1:\Gamma_g\to \Gamma_f$. Stable efficiency tells us that this map is surjective, and we can extend it to a continuous and surjective map $\hat i_1:\X_g \to \X_f$. Naturally, this map is a semi-conjugacy between $\hat g$ and $\hat f$. For the remaining semi-conjugacy, we consider $\Gamma_f^1$ the family of curves associated to $(f, f(\D))$ and $(\X_f^1,\hat f_1)$ its natural tree model.  For every $\gamma\in \Gamma_f^1$, there exists $\gamma_g\in \Gamma_g$ such that $\gamma\subset \gamma_g$. This defines a map $i_2:\Gamma_f^1\to \Gamma_g$ that we use to construct a semi-conjugacy between $(\X_f^1,\hat f_1)$ and $(\X_g,\hat g)$. Since $(\X_f^1,\hat f_1)$ is conjugate to $(\X_f,\hat f)$, we naturally have the final semi-conjugacy between $\hat f$ and $\hat g$. 

\begin{figure}[h]
\centering
\input{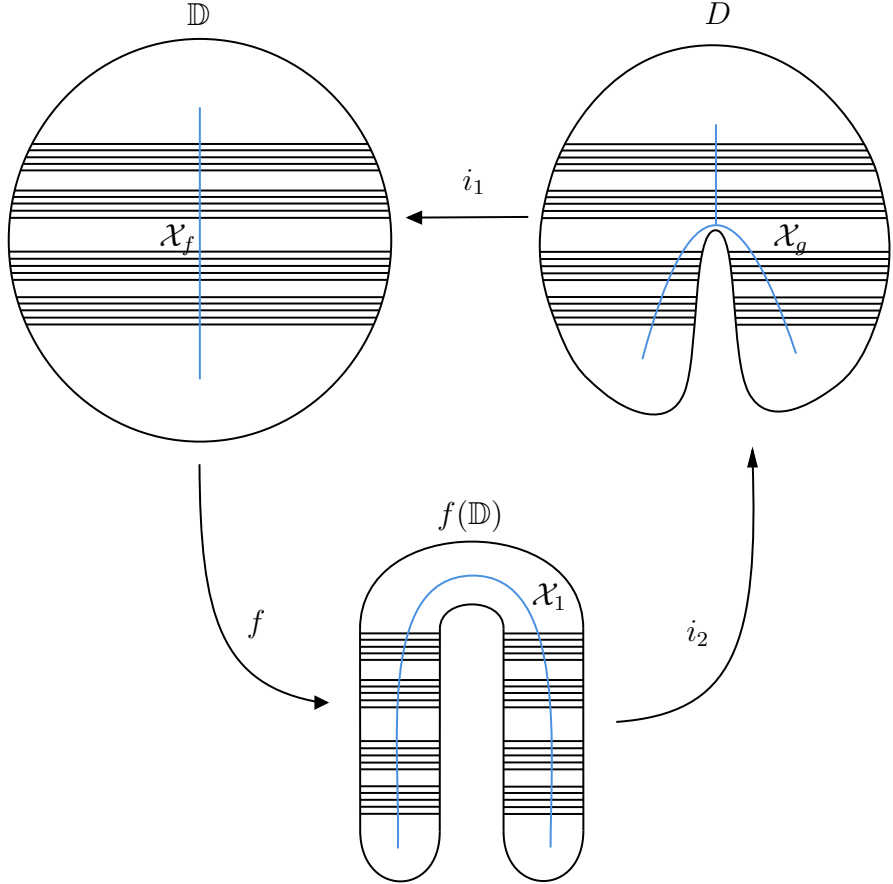}
\caption{Inclusion of curves.}\label{ArvDinConMaxInv}				
\end{figure}

We now aim to prove Theorem \ref{TeoSemimu}.

Suppose we are in the hypothesis of Theorem \ref{TeoSemimu}. As in the proof of Proposition  \ref{ConjDisco}, we consider $R_{f,\mu}\subset \Lambda_f$ and $R_{g,\nu}\subset \Lambda_g$ subsets of regular points verifying $\mu(R_{f,\mu})=1$, $\nu(R_{g,\nu})=1$, and $h(R_{f,\mu})=R_{g,\nu}$. We fix $x\in R_{f,\mu}$.  Observe that for $\gamma\in \Gamma_{\mu,x}$, $h(\gamma)$ no longer makes sense, $\gamma$ includes points that do not belong to $\Lambda_f$. Therefore, we need to consider $h(\gamma\cap\Lambda_f)$. We would like that for every $\gamma\in \Gamma_{\mu,x}$, there exists $\tilde \gamma\in \Gamma_{\nu,h(x)}$ such that $h(\gamma\cap\Lambda_f)\subset \tilde \gamma$. With this, we could define an inclusion as in the example and from there extend it to the tree as a semi-conjugacy. Unfortunately, this does not always happen. In the example of the horseshoe, the curves in the classical construction are not contained in the curves for the cut disk. However, if we iterate $f$ enough, we do have this property. 

For each $k\geq 1$, consider $(\X^k,\hat f_k)$ the natural tree model associated to $f_{|f^k(\D)}\in \text{MD}^r(f^k(\D))$. Observe first that  $R_{f_{|f^k(\D)},\mu_{|f^k(\D)}}= R_{f,\mu}$  and consider $\Gamma^k_{\mu,x}$ the family of curves that define $\X^k$.

\begin{lema}\label{LemCurvas} There exists $k>0$ such that for every $\gamma\in \Gamma^k_{\mu,x}$, there exists $\tilde\gamma \in \Gamma_{\nu,h(x)}$ for which $h(\gamma\cap\Lambda_f)\subset \tilde \gamma$.  Moreover, for every $\tilde \gamma\in \Gamma_{\nu,h(x)}$, there is at least one $\gamma \in \Gamma^k_{\mu,x}$ that verifies  $h(\gamma\cap\Lambda_f)\subset \tilde \gamma$.
\end{lema}

\begin{proof}

Recall that $\Gamma^k_{\mu,x}$ is defined as a union of families of curves constructed inductively. We shall represent these families by $\Gamma^k_{\mu,i}$, with $i\in\N$. 

First, we will find  $k$ such that for every $n\in \Z$, 
\[h(W^s_{f^k(\D)}(f^n(x))\cap \Lambda_f)\subset W^s_\D(h(f^n(x))).\] 
This proves the thesis of the lemma for every curve in $\Gamma^k_{\mu,0}$.

Note that since  $h$ is uniformly continuous in $\Lambda_f$, given $\e>0$, there exists $\delta>0$ such that $d(h(y),h(z))< \e$ for every $y,z\in \Lambda_f$ when $d(y,z)<\delta$. This implies that for any subset $W\subset W^s_\delta(x)$ (see Theorem \ref{teoStaMan} to recall the definition of $W^s_\delta(x)$), we know $h(W\cap \Lambda_f)\subset W^s_\e(h(x))$. This assertion holds true beyond $x$ being or not being a regular point of $f$. 

Let us fix a hyperbolic block $B$ which $x$ belongs to. Consider $\{n_i\}_{i\in \Z}=\{n\in \Z:f^n(x)\in B\}$.   We claim that for any $\delta$, there exists $k$ such that $W^s_{f^k(\D)}(f^{n_i}(x))$ has a length smaller than $\delta$ for every $i\in\Z$. Otherwise, there exists $\delta$, such that for every $k\in \N$, there exists $n_{i_k}$, where $ W^s_{f^k(\D)}(f^{n_{i_k}}(x))$ has a length greater than $\delta$. Now, taking a sub-sequence, if necessary, we find a point $y\in B$ (because $B$ is compact) that has a segment of the stable manifold contained in $\Lambda$, an absurd. By the contraction property in the stable manifold theorem, we extend our claim to every $n\in \Z$. This means that if we fix $\delta$, there exists $k$ such that $W^s_{f^k(\D)}(f^{n}(x))$ has a shorter length than $\delta$ for every $n\in\Z$.

The above implies that $W^s_{f^k(\D)}(f^n(x)) \subset W^s_\delta(f^n(x))$. 

By the stable manifold theorem applied to $h(x)$ and $\nu$, there exists $\e>0$ such that for every $n\in \Z$, $ W^s(g^n(h(x)))\cap W^s_\e(g^n(h(x)))\subset W^s_\D(g^n(h(x)))$.

Putting everything together, we conclude that for every $n\in \Z$,  
\[h(W^s_{f^k(\D)}(f^n(x))\cap \Lambda_f)\subset W^s(h(f^n(x)))\cap W^s_\e( h(f^n(x)))\subset W^s_\D(h(f^n(x))).\] 

We now have to extend this to $\Gamma^k_{\mu,i}$ for $i\geq 1$ (the other levels in the construction of $\Gamma_{\mu,x}$). For this, we may shrink $\e$ and therefore $\delta$ such that the following property holds: if $\gamma_1,\gamma_2\in \Gamma_{\nu,h(x)}$ are such that $g(\gamma_1)$ and $g(\gamma_2)$ belong to the same $\gamma\in \Gamma_{\nu,h(x)}$, then $d( g(\gamma_1),g(\gamma_2))\geq \e$. This assertion holds because $g$ is continuous, $\Lambda_g$ is compact, and  the distance from $\Lambda_g$ to the border of the disc is positive. 

 Now, if we take $\gamma_1\in \Gamma^k_{\mu,1}$,  $f(\gamma_1)$ is contained in some $\gamma_0 = W^s_{f^k(\D)}(f^n(x))$, and there is a unique $\tilde \gamma_1\in \Gamma_{\nu,h(x)}$ such that $h(f( \gamma_1)\cap \Lambda_f)\subset g(\tilde \gamma_1)\subset $  $W^s_\D(h(f^n(x)))$. Therefore, $h(\gamma_1\cap \Lambda_f)\subset \tilde \gamma_1$. This reasoning extends the property to $\Gamma^k_{\mu,1}$, and by induction, we deduce the claim for every $\gamma\in \Gamma^k_{\mu,x}$. 
 
 The final property of the lemma can be checked manually in each step of the construction of $\Gamma_{\nu,h(x)}$. 
\end{proof}

We are now in condition to prove Theorem \ref{TeoSemimu}.

\begin{proof}[of Theorem \ref{TeoSemimu}]
By Lemma \ref{LemTreeK}, we know  $\gamma \in \Gamma_{\mu,x}\mapsto f^k(\gamma)\in \Gamma^k_{\mu,x}$ is a bijection. By Lemma \ref{LemCurvas}, we construct a function $\gamma\in\Gamma^k_{\mu,x}\mapsto \tilde \gamma\in\Gamma_{\nu,h(x)}$. The composition of these two functions gives us a map $H:\Gamma_{\mu,x}\to \Gamma_{\nu,h(x)}$ that verifies $h(f^k(\gamma)\cap \Lambda_f) \subset H(\gamma)$. Since $\Gamma_{\mu,x}\subset \X_{f,\mu}$ and $\Gamma_{\nu,h(x)}\subset \X_{g,\nu}$, to prove Theorem \ref{TeoSemimu}, we will extend $H$ to the whole tree and check that our desired properties are verified. 

Since $f^k$ is a diffeomorphism between $\D$ and $f^k(\D)$, we can suppose without loss of generality that $k=0$. Given $\hat x\in \X_{f,\mu}$, take $(s_n)_{n\in\N}$ as a sequence of compact connected surfaces such that $\hat x$ is the class of $(s_n)$ according to the definition of $\sim$ in subsection \ref{subsectreemodel}. Since $\hat x\in \X_{f,\mu}$, it must be a minimal element of the natural partial order.  We now enumerate the properties $(s_n)$ verifying the following: 
\begin{enumerate}
\item  $\overline{s_{n+1}}\subset \interior(s_n)$
\item  The boundary of each $s_n$ consists of a finite number of curves from $\Gamma_{\mu,x}$ and $\partial\D$.
\item There is at most one $\gamma\in \Gamma_{\mu,x}$ that verifies $\gamma\subset s_n$ for all $n$. 
\end{enumerate}

For each $s_n$, let us consider $\gamma^n_1,\cdots, \gamma^n_{k_n} \subset \Gamma_{\mu,x}$ the curves that form the border of $s_n$ in the interior of the disk. To simplify our notation, we shall call the collection of these curves as $\partial s_n$.  Now, with the set of curves $H(\gamma^n_1),\cdots, H(\gamma^n_{k_n})$, we can construct a finite number of compact surfaces in $\D$ having some of these curves as the border. Let us call $r_1,\cdots, r_l$ the surfaces of this type that contain $h(s_n \cap \Lambda_f)$. We define $H_1(s_n)=\cap^{l}_{i=1} r_i$, which is in fact one $r_i$ and moreover the smallest one. We call $\partial H_1(s_n)$ the collection of curves that form the border of $H_1(s_n)$. This collection is a subset of $\{H(\gamma^n_1),\cdots, H(\gamma^n_{k_n})\}$.

By definition, $(H_1(s_n))_{n\in\N}$ is a sequence of compact surfaces whose boundaries consist of a finite number of curves from $\Gamma_{\nu,y}$ and $\D$. It is easy to observe that $\overline{H_1(s_{n+1})}\subset \interior(H_1(s_n))$. However, this first pick of $H_1(s_n)$ may not represent an element of $\X_{g,\nu}$, because it may not be minimal. We must first refine $H_1(s_n)$. We can construct $H(s_n)\subset H_1(s_n)$   that verifies the following:
\begin{enumerate}
\item $\overline{H(s_{n+1})}\subset \interior(H(s_n))$
\item The boundary of each $H(s_n)$ consists of a finite number of curves from $\Gamma_{\nu,y}$ and $\partial\D$. Moreover,  $\partial H_1(s_n)\subset \partial H(s_n)$. 
\item If a curve $\tilde \gamma \in \Gamma_{\nu,h(x)}$ verifies $\tilde \gamma \subset H(s_n)$ for all $n$, then $\tilde \gamma \cap h(\Lambda_f\cap s_n)\neq \emptyset$. 
\end{enumerate} 

This is possible by simply removing regions of $H_1(s_n)$ such that  property (3) is verified, and, with some care, property (1) holds. Let us prove that this sequence of surfaces is minimal. We must show that there is at most one $\tilde \gamma\in \Gamma_{\nu,h(x)}$ that verifies $\tilde \gamma\subset H(s_n)$ for all $n$. Suppose $\tilde \gamma\in \Gamma_{\nu,h(x)}$ verifies that. By property (3), we know that $\tilde \gamma \cap h(\Lambda_f\cap s_n)\neq \emptyset$. Therefore, $h^{-1}(\tilde \gamma \cap \Lambda_g)\cap s_n$ is contained in at least one curve $\gamma\in \Gamma_{\mu,x}$. Moreover, it is contained in at most a finite number of curves $\gamma_1,\cdots, \gamma_k\in \Gamma_{\mu,x}$. By the compactness of every $s_n$, there is at least one $\gamma_i\subset s_n$ for every $n$, and therefore there is a unique $\gamma$. From this, if $\tilde \gamma_1\neq \tilde \gamma_2$ verifies $\tilde \gamma_i \subset H(s_n)$ for all $n$, then $\gamma_1$ and $\gamma_2$ must be different, a contradiction with property (3) of $(s_n)$.

Our previous construction allows us to define the map $\hat h:\X_{f,\mu}\to \X_{g,\nu}$ by $\hat h(\hat x)$ as the class of the sequence $(H(s_n))_{n\in\N}$ in $X_{g,\nu}$.  This map naturally verifies $\hat h( \gamma) = H(\gamma)$ and is inherently continuous by our construction.  Let us check that $\hat h$ conjugates the dynamics. Since $\hat h$ is continuous, we only need to check it in $\Gamma_{\mu,x}$. Consider $\gamma,\tilde \gamma\in \Gamma_{\mu,x}$ and $\tilde H(\gamma)\in \Gamma_{\nu,y}$ such that $f(\gamma)\subset \tilde \gamma$ and $g(H(\gamma))\subset \tilde H(\gamma)$. Then, 
\[\hat h (\hat f(\gamma))= H(\hat f(\gamma))= H(\tilde \gamma) =  \tilde H(\gamma) =  \hat g (H(\gamma))=\hat g (\hat h (\gamma)),\]
and therefore $\hat h\circ \hat f = \hat g \circ \hat h$.

The map $\hat h$ is surjective because $h(\Lambda_f)=\Lambda_g$ and since $g$ is stable efficient, we know that $\pi_{g,\nu}(\Lambda_g)=\X_{g,\nu}$. 

For the semi-conjugacy in the other direction, we switch positions of $f$ and $g$ in all of our previous arguments.   
\end{proof}

\begin{proof}[of Theorem \ref{TeoSemi}]
We observe that in Lemma \ref{LemCurvas}, for every $\hat k\geq k$,  the thesis is also verified. If $\mu_1,\cdots, \mu_l$ are the measures associated to the finite ergodic covering property, we can choose a $k$ in Lemma \ref{LemCurvas} common to these measures. Then, the proof of this result follows as the proof of Theorem \ref{TeoSemimu}.
\end{proof}

\section{Proof of Theorem \ref{TeoConvErg}}\label{SecConv}

Last, we prove Theorem \ref{TeoConvErg}. Our first step is to thicken the family of curves constructed in section \ref{treeConstruction}.
Let us recall that given $f\in \text{MD}^r(\D)$ and $\mu\in \M_a(f)$, we picked a regular point $x\in R_{f,\mu}$, defined $\Gamma_{\mu,x}$, and then defined $\Gamma=\cup_{\mu \in \M_a(f)}\Gamma_{\mu,x}$. For this part, we need to work with a family of curves $\Gamma_\mu$ that also define the tree $\X_f$ and verify  $\pi_*(\mu)(\Gamma_\mu)=1$. Since $\Gamma_{\mu,x}$ is countable, $\pi_*\mu(\Gamma_{\mu,x})=0$; therefore, we must thicken it. For this, we define
\[\Gamma_\mu = \bigcup_{x\in R_{f,\mu}}\Gamma_{\mu,x}\]
and redefine $\Gamma=\bigcup_{\mu \in \M_a(f)} \Gamma_\mu$.
Property P4* of the former $\Gamma$ implies that the tree induced by our new $\Gamma$ is the same one as before. We easily deduce the following lemma:

\begin{lema}\label{Enchimento}
Given $f\in \text{MD}^r(\D)$  and $\mu\in \M_a(f)$,  $\pi_*(\mu)(\Gamma_\mu)=1$ and $\pi_{*}(\mu)(\Gamma)=1$ hold true.
\end{lema}
\begin{proof}
The set of regular points $\mu$ have full measure.
\end{proof}

\begin{proof}[of Theorem \ref{TeoConvErg}]

Let $f,g\in \text{MD}^r (\D)$ be stable efficient and consider $(\X_f,\hat{f})$ and $(\X_g,\hat{g})$ the natural tree models and $\Gamma_f$ and $\Gamma_g$ the family of curves associated to $f$ and $g$ respectively. Suppose there exist $Y_f\subset \X_f$ and $Y_g\subset \X_g$  and  
$\hat h: Y_f\to Y_g$ such that:
\begin{itemize}
\item $Y_f$ and $Y_g$ are $f$ and $g$ invariants respectively. 
\item $Y_f$ and $Y_g$ have full measure for every aperiodic ergodic measure of $\hat{f}$ and $\hat{g}$.
\item $\hat h$ is a measurable bijection. 
\item    $\hat{g}\circ \hat h= \hat h\circ \hat{f}_{|Y_f}$
\end{itemize}

We will prove the theorem using the inverse limits of a dynamical system. Therefore, we shall establish the background before diving into the proof. We take $\varprojlim\X_f$ and $\varprojlim\X_g$ as the inverse limits of $(\X_f,\hat{f})$ and $(\X_g,\hat{g})$ respectively. For instance, in the case of $\X_f$, it is the space of sequences in $\X_f$ indexed in the non-positive integers $[\cdots,\hat{x}_{-n},\cdots,\hat{x}_0]$ such that $\hat f (\hat x_{-n}) = \hat x_{-n+1}$ for all $n\geq 1$. In said space, we define two maps: 1) the projection $\Pi_f:\varprojlim\X_f\to \X_f$ defined by $\Pi_f([\cdots,\hat{x}_{-n},\cdots,\hat{x}_0]) = \hat{x}_0$ and 2) the dynamical system $\hat{F}: \varprojlim\X_f \to \varprojlim\X_f$ defined by
\[\hat F([\cdots,\hat{x}_{-n},\cdots,\hat{x}_0])= [\cdots,\hat{x}_{-n},\cdots,\hat{x}_0, \hat{f}(\hat{x}_0)].\]
The relationship between the dynamics of $\hat f$ and $\hat F$ is given by the equation $\hat f \circ \Pi_f = \Pi_f \circ \hat F$.

We now relate the dynamics of $f$ with the dynamics of $\hat F$ with the map $\phi_f:\Lambda_f \to \varprojlim\X_f$. Given $x\in \Lambda_f$, take $\hat{x}_{-n}=\pi_f(f^{-n}(x))$ and then
$\phi_f(x) =( [\cdots,\hat{x}_n,\cdots,\hat{x}_0])$.  It is easy to observe that $\hat F \circ \phi_f = \phi_f \circ f$. We would now like for $\phi_f$ to be a measurable bijection. The map $\phi_f$ is surjective because $\pi_f(\Lambda_f)=\X_f$ and $\hat f\circ \pi_f = \pi_f \circ f$. Although it may fail to be injective, we can show that it is injective in a set of full measure. Given $x$ in $\Lambda_f$, we define $[x]$ as the connected component of $\pi_f^{-1}{(\pi_f(x))}\cap \Lambda_f$ that contains $x$ and the set 
\[N_f=\{x\in\Lambda_f:\ \card([x])=1\text{ and }x\in\gamma\text{ for some }\gamma\in \Gamma_f\}.\]

We claim that $\phi_f$ is injective when restricted to $N_f$. If $x_1,x_2\in N_f$ with $\phi_f(x_1)=\phi_f(x_2)$,  then $f^{-n}(x_1)$ and $f^{-n}(x_2)$ are not separated by any curve in $\Gamma_f$, for every $n\geq 0$. Take $\gamma_{-n}\in \Gamma_f$ such that $f^{-n}(x_1),f^{-n}(x_2)\in \gamma_{-n}$, and consider the segment $\hat \gamma_{-n}$ in $\gamma_{-n}$ whose endpoints are $f^{-n}(x_1)$ and $f^{-n}(x_2)$. Our previous observation implies that $f(\hat \gamma_{-n})= \hat \gamma_{-n+1}$ for all $n\geq 1$. Therefore, $\hat \gamma_0= f^n(\hat \gamma_{-n})\subset f^{n}(\D)$, $\hat \gamma_0 \subset \Lambda_f$, and finally $[x_1]=[x_2]$. From this, we deduce that $x_1=x_2$. 

Let us check that $\mu(N_f)=1$ for each $\mu\in \M_a(f)$. Given a regular point $x$ for $\mu$, define $\gamma_{-n}=W_{\D}^s(f^{-n}(x))$ and observe that $[x]=\bigcap_{n\geq 0} f^n(\gamma_{-n})$ by a similar reasoning as in the previous paragraph. Moreover, for regular points it is verified that $\bigcap_{n\geq 0} f^n(\gamma_{-n}) = \{x\}$ and therefore $card([x])=1$. In fact, for any segment $\hat \gamma\subset W^s_\D(x)$ that contains $x$ in its interior, there exists $n\geq 0$ such that $\D\setminus f^{-n}(\hat \gamma)$ has two connected components. Then, $\gamma_{-n}\subset f^{-n}(\hat \gamma)$ and therefore $[x]\subset \hat \gamma$ for any $\hat \gamma$, thus $[x]=\{x\}$. 

From our previous reasoning, we can infer that the map $\phi^{-1}:\phi_f(N_f)\to N_f$ is defined by $\phi_f^{-1}([\cdots,\hat{x}_{-n},\cdots,\hat{x}_0])$ is the only point in the set $\bigcap_{n\geq 0} f^n(\gamma_{-n})$ where $ \gamma_{-n}$ is the curve in $\Gamma_f$ such that $\gamma_{-n} = \hat{x}_{-n}$.

 Since $\Pi_f \circ \phi_f= \pi_f$, the following diagram is commutative:
\[\xymatrix{
	\M_1(f) \ar[r]^{\phi_{f*} }\ar[rd]_{\pi_{f*}} &\M_1( \hat{F} ) \ar[d]^{\Pi_{f*}} \\
	\  & \M_1( \hat{f} ).
}\]

Observe now that  Proposition \ref{PropProp} and Theorem \ref{teoCP01} implies that  $\pi_{f*}:\M_a(f)\to \M_a(\hat f)$ is a bijection. Also, it is known that $\Pi_{f*}:\M_1( \hat{F} ) \to \M_1( \hat{f} ) $ is a bijection. Therefore, $\phi_{f*}$, when restricted to $\M_a(f)$, is a bijection toward $\M_a(\hat{F})$. 

Let us consider for $g$ the objects $\Pi_g,\ \hat G,\ \phi_g$, and $N_g$. We take $\hat Y_f = \Pi_f^{-1}(Y_f)\subset\varprojlim\X_f$ and  $\hat Y_g = \Pi_g^{-1}(Y_g)\subset\varprojlim\X_g$ and define $\hat H: \hat Y_f \to \hat Y_g$ as the induced map of $\hat h$ by
\[	\hat H([\cdots,\hat{x}_{-n},\cdots,\hat{x}_0]) = [\cdots,\hat{h}(\hat{x}_{-n}),\cdots,\hat{h}(\hat{x}_0)].\]

Since $\hat h_{*}:\M_a(\hat f)\to \M_a(\hat g)$ is a bijection, so is the map $\hat H_{*}:\M_a(\hat F)\to \M_a(\hat G)$.

Let us consider $B_f = \phi_f(N_f)\cap \hat Y_f$ and $B_g= \phi_g(N_g)\cap \hat Y_g$. Define $A_f= B_f\cap \hat H^{-1}(B_g)$ and $A_g= B_g\cap \hat H(B_f)$. It is easy to observe that $B_f$ has full measure for every $\mu\in \M_a(\hat F)$ and $B_g$ has full measure for every $\mu \in \M_a(\hat G)$. With this, we see that $A_f$ and $A_g$ also have full measure for every aperiodic measure in their respective context. Define $M_f = \phi_f^{-1}(A_f)$, $M_g = \phi_g^{-1}(A_g)$, and $h:M_f\to M_g$ by 
\[h(x) = \phi_g^{-1}\circ \hat H \circ \phi_f(x).\]

Since  $\phi_f$, $\hat{H}$, and  $\phi_g^{-1}$ are measurable bijections when restricted to $M_f$, $A_f$, and $A_g$, the map $h$ is also a measurable bijection. The inverse $h^{-1}:M_g\to M_f$ defined by $h^{-1}(x) = \phi_f^{-1}\circ \hat H^{-1} \circ \phi_g(x)$ is also measurable.

Since the following diagram is commutative 
$$\xymatrix{
	\mathcal{M}_a(f)   \ar[r]^{h_*}\ar[d]_{\pi_f*}\ar@/_9mm/[dd]_{\phi_f*} &\mathcal{M}_a(g) \ar[d]^{\pi_g*}\ar@/^9mm/[dd]^{\phi_g*} \\
	\mathcal{M}_a(\hat{f})	\ar[r]^{\hat{h}_*}  & \mathcal{M}_a(\hat{g})  \\
	\mathcal{M}_a(\hat{F})\ar[u]_{\Pi_{f}*}\ar[r]^{\hat{H}_*}& \mathcal{M}_a(\hat{G}) \ar[u]^{\Pi_{g}*}  
}$$
and we know that each arrow, except $h_*$, represents bijection, we deduce that $h_*$ is also a bijection.
\end{proof}

	
\bibliographystyle{siam}
\bibliography{Bibliog}

\end{document}